\documentclass[10pt,a4paper]{article}
\usepackage[english]{babel} 
\usepackage[utf8]{inputenc}
\usepackage{amsmath}
\usepackage{amsfonts}
\usepackage{amssymb}
\usepackage{amsthm}
\newtheorem{theorem}{Theorem}
\newtheorem{corollary}{Corollary}
\newtheorem{definition}{Definition}
\newtheorem{proposition}{Proposition}
\title{Arithmetic volumes of indefinite ternary lattices with square free discriminant}
\author{Ekaterina Stuken}
\begin{document}
\maketitle
\begin{abstract}
The main purpose of this paper is to give an explicit formula for the volumes of indefinite ternary lattices with square free discriminant. We also consider several examples of lattices, for which the fundamental domain is well-known and make sure the results obtained by our computation coincide with the known ones.
\end{abstract}

\section*{Introduction}

For the detailed description of everything mentioned here we refer to \cite{con}, \cite{gri}. 

Let (V,Q) be a quadratic space over $\mathbb{Q}$, dim(V)=$n\geq 3$ and (r,s) be the signature of $(\mathrm{V}\otimes \mathbb{R},\mathrm{Q})$. Also, O(r,s) denotes the orthogonal group $\mathrm{O}(\mathrm{V}\otimes \mathbb{R},\mathrm{Q})$ and $\mathrm{SO_0(r,s)}$ denotes the connected component of O(r,s). The quotient space
$$ X=\mathrm{O(r,s)/O(r)\times O(s)=SO_0(r,s)/SO(r)\times SO(s)}$$
is isomorphic to a non-compact symmetric space of dimension rs.

Consider the compact symmetric space $X^V=\mathrm{O(r+s)/O(r)\times O(s)}$ dual to $X$. It is well-known that $X$ embeds into the compact dual $X^V$ as an open domain (\cite{hel}).

Let $L$ be an integral lattice in (V, Q) (that is $(x,y)=\frac{1}{2}(Q(x+y)-Q(x)-Q(y))\in \mathbb{Z} \ \forall x,y\in L$). $\mathrm{O(L)}=\lbrace g\in \mathrm{O(r,s)} \ |\ gL=L \rbrace$ denotes an arithmetic discrete subgroup of $\mathrm{O(r,s)}$ and $\mathrm{O(L)\setminus O(r,s)}$ and $\mathrm{O(L)}\setminus X$ have finite volumes for the $\mathrm{O(r,s)}$-invariant volume forms. Let us choose the volume form $\mathrm {Vol()}$ on the domain $X$ and on the compact dual $X^V$ that coincide at the common point $\mathrm{O(r)} \times \mathrm{O(s)}$ of both domains.


Consider two invariants assigned to the lattice $L$.


\begin{definition}
The normalised Hirzebruch-Mumford volume of the lattice $L$ is $$\mathrm {Vol_{HM}(O}(L))=\frac{\mathrm {Vol(O}(L)\setminus X)}{\mathrm {Vol}(X^V)}$$
\end{definition}
It is the first invariant.

\begin{definition}
Let $S\in \mathrm {Mat_{n\times n}}(\mathbb{Q})$ be the Gram matrix of $L$, $p$ - prime number. The local density of $L$ is 
$$\alpha_p (S) = \frac{1}{2} \lim_{r \to \infty} p^{-\frac{rn(n-1)}{2}} \mid \{ X\in \mathrm {Mat_{n\times n}}(\mathbb{Z}_p) \  \mathrm{mod} \ p^r;\  X^t SX \equiv S \  \mathrm{mod} \ p^r\}\mid$$
\end{definition}
The second invariant is $\prod \limits_p \alpha_p(L)^{-1}$ where the product is taken over all primes $p$




%
Siegel found the connection between these two invariants.

\begin{theorem} [Siegel's formula] \cite{gri}, \cite{sie} 
Let $L$ be an indefinite lattice of rank $\rho$ $\geq$ 3, $g_{sp}$ -- the number of spinor genera in the genus of $L$. Then the normalised volume of $\mathrm O(L)$ is equal to $$\mathrm{
Vol_{HM}(O}(L))=\frac{2}{g_{sp}} \cdot \mid \mathrm {det}L \mid
^{(\rho+1)/2} \prod \limits_{k=1}^{\rho} \pi^{-k/2} \Gamma (k/2)
\prod \limits_p \alpha_p(L)^{-1}$$
\end{theorem}

\begin{theorem} [Kneser] \cite{kne} \\
Let $L$ be an indefinite lattice. Suppose the genus of $L$ contains more than one class. Then there exists a prime $p$ such that the quadratic form of $L$
can be diagonalised over $\mathbb{Z}_p$, and all the diagonal entries consist of different powers of $p$.
\end{theorem}

\begin{corollary}
Suppose $\mathrm{rank}(L)\geq 3$ and $\mathrm{det}(L)$ is square free. Then the genus of $L$ coincides with the class of $L$. In particular, $g_{sp}=1$.
\end{corollary}

\section*{Computation of densities}

In \cite{kit} we can find the following theorems.

Let $L=\bigoplus_{j\in \mathbb{Z}} L_j$, $L_j$ is a $p^j$-modular
 lattice of rank $n_j\in \mathbb{Z}_{\geq 0}$. Let $$\omega:=\sum \limits_{j} j n_j ((n_j+1)/2+\sum
 \limits_{k>j}
 n_k)\ ,\ P_p(n):=\prod \limits_{i=1}^{n} (1-p^{-2i})$$
$$ \chi(W) := \begin{cases}
0,&\text{if dim W is dd}\\
1,&\text{if W is a hyperbolic space}\\
-1,&\text{otherwise}
\end{cases}$$
\begin{theorem}
Let $p\neq 2$, then
 $$ \alpha_p(L)=2^{s-1}p^\omega P_p(L)E_p(L)$$
where $s$ is the number of non-zero $p^j$-modular term of $L_j$ in
Jordan decomposition of $L$,
$$ P_p(L)=\prod \limits_{j} P_p([n_j/2]),\ E_p(L)=\prod \limits_{j,L_j\neq 0}
(1+\chi(N_j)p^{-n_j/2})^{-1},$$ $N_j$ - unimodular lattices corresponding to $L_j$
(it means that $L_j=N_j(p^j)$).
\end{theorem}
Each unimodular lattice over $\mathbb{Z}_2$ is an orthogonal sum $N=N^{\mathrm{even}}\oplus N^{\mathrm{odd}}$ of even and odd sublattices. Here $\mathrm{rank}(N^{\mathrm {odd}})\leq 2$. $N^{ \mathrm{even}}$ is a sum of several hyperbolic planes and, perhaps, one lattice $\begin{pmatrix} 2 & 1 \\ 1 & 2 \end{pmatrix}$.
\begin{theorem}
Let $q=\sum \limits_{j} q_j$,
$$ q_j = \begin{cases}
0,&\text{if $N_j$ is even}\\
n_j,&\text{if $N_j$ is odd and $N_{j+1}$ is even}\\
n_j+1,&\text{if $N_j$ and $N_{j+1}$ are odd}
\end{cases}$$
$$ P_2(L)=\prod \limits_{j} P_2(\mathrm{rank}(N_j^{\mathrm{even}}/2)), \ E_2(L)=\prod \limits_{j}
E_j^{-1}$$ where
$$ E_j=\frac{1}{2}(1+\chi (N_j^{\mathrm{even}})2^{-\mathrm{rank}(N_j^{\mathrm{even}})/2})$$
if both $N_{j-1}$ and $N_{j+1}$ are even, except the case when
$N_j^{\mathrm{odd}}\cong\langle \varepsilon_1 \rangle \oplus \langle
\varepsilon_2 \rangle, \ \varepsilon_1 \equiv \varepsilon_2 \ mod
\ 4$. In all other cases let $E_j=1/2$. 
Then the local density $\alpha_2(L)$ is equal to:
$$\alpha_2(L)=2^{n-1+\omega-q}P_2(L)E_2(L)$$
\end{theorem}

\section*{Computation of volume of indefinite ternary lattice with square free discriminant}
\begin{definition} (Watson moves $E_p(L),\ F_p(L)$)\\
$E_p(L)=L+p^{-1}L\cap pL^* \subset L\otimes \mathbb{Q}$, where $L^*$ is the dual lattice of $L$. \\
$F_p(L)=E_p(\sqrt{p}L)$.
\end{definition}
For the detailed description of Watson moves we refer to \cite{vin}. Notice, that the scalar products of vectors in $E_p(L)$ are integer, so $E_p(L)$ -- quadratic lattice, and $\mathrm{O}(L)\subset \mathrm{O}(E_p(L))$.
\begin{proposition}
Each indefinite ternary lattice embeds into a maximal lattice with square free discriminant.
\end{proposition}
\begin{proof}
Let $L=L_0\oplus[p]L_1\oplus[p^2]L_2\oplus\ldots$ be the Jordan decomposition of $L$.
$(E_p(L))_q=L_q,\ q\neq p$, $(E_p(L))_p=L_0\oplus [p]L_1 \oplus (L_2\oplus [p]L_3\oplus \ldots)$. So, in each $p$-th component starting from $p^2$ the degree of the corresponding invariant factor decreases by $2$. Now let us apply sufficient number of operations $E_p$ for every $p$ dividing the discriminant of $L$. We obtain that the lattice $L$ is embedded into $L'$, and all invariant factors of $L'$ are square free. If the discriminant of $L'$ is not square free yet, there exists $p$ such that $\mathrm {rank}(L_1)\geq 2$.
Now let us apply $F_p$ to $L'$.  $(F_p(L))_q=L_q,\ q\neq p$. $(F_p(L))_p=L_1\oplus [p]L_0$. $F_p(F_p(L))=L$, so $\mathrm{O}(L)=\mathrm{O}(F_p(L))$. So, after applying sufficient number of $F_p$, we get the desired lattice.
\end{proof}

\begin{theorem}
Let $L$ be an indefinite ternary lattice with square-free discriminant $d$. The normalised volume of $L$ is equal to
$$ V=\begin{cases}
  \frac{1}{3\cdot 2^{\omega(d)+2}} \prod \limits_{p | d, p\neq 2} (p \pm 1 ) & \text{for even lattices, }2|d\\
   \frac{1}{ 2^{\omega(d)+3}} \prod \limits_{p | d, p\neq 2} (p \pm 1 )  & \text{for odd lattices, }2|d \\
   \frac{2\pm 1}{3\cdot 2^{\omega(d)+4}} \prod\limits_{p|d} (p\pm 1)& \text{for } 2\nmid d
 \end{cases} $$
where $\omega(d)$ is a number of prime divisors of $d$. The sign plus for $p\neq 2$ is chosen iff the corresponding unimodular component represents $0$ over $\mathbb{Z}_p$. The sign plus in the last formula for $p=2$ is chosen iff $d=-1 \mod 4,\ \varepsilon_2(L)=1$ or $d=1 \mod 4, \ \varepsilon_2(L)=-1$ ( $\varepsilon_2(L)$ -- Hasse invariant over $\mathbb{Q}_2$).
\end{theorem}
\begin{proof}
By corollary 1 the genus of $L$ coincides with it's spinor genus and class, so $g_{sp}=1$
$$ V =  2 d^2 \prod\limits_{k=1}^3 \pi^{-k/2}\Gamma(k/2)\prod a_p^{-1} = \frac{d^2}{\pi^2}\prod \alpha_p^{-1}$$

We study the Jordan decomposition of $L$ over $\mathbb{Z}_p$.

For $p\nmid d$ the lattice is unimodular of rank $3$, so $\alpha_p=(1-p^{-2})$.

Consider the case $p|d$, $p\neq 2$. 
$L=L_0\oplus[p]L_1$, $\mathrm{rank}(L_0)=2,\ \mathrm{rank}(L_1)=1$
The local densities $\alpha_p$ are equal to 
$$ a_p=2p(1-p^{-2})(1+\chi (L_0) p^{-1})^{-1}$$

$\chi(A)=1$, if $A\sim \begin{pmatrix}
0 & 1\\
1 & 0
\end{pmatrix}$, $\chi(A)=-1$ otherwise.

$L_0\sim \begin{pmatrix}
0 & 1\\
1 & 0
\end{pmatrix}$ iff the quadratic form, corresponding to $L_0$, represents $0$. It happens iff $\mathrm{det}(L_0)=-1$.

So,
$\alpha_p=2p(1-p^{-2})(1 \pm p^{-1})^{-1}$, the sign plus is chosen when the quadratic form represents $0$, otherwise the minus sign is chosen. 

Consider the case of even $d$.

The Jordan decomposition of $L$ is 
$ \begin{pmatrix}
a & b &  \\
b & c &  \\
    &  &  2\epsilon
\end{pmatrix} $, where in the upper left corner there is an unimodular component of rank $2$, $\epsilon$ is a unit in $\mathbb{Z}_p$. 

If this unimodular component is odd, then $\mathrm{rank}(N_0^{\mathrm{even}})=\mathrm{rank}(N_1^{\mathrm{even}})=0$, $E_{-1}=E_0=E_1=E_2=\frac{1}{2}$, so $\alpha_2=2^{3-1+1-4}2^4=2^3$.

If this unimodular component is even, then $E_0=E_2=\frac{1}{2}$, $E_1=1$, $\mathrm{rank}(N_0^{\mathrm{even}})=2$, $\mathrm{rank}(N_1^{\mathrm{even}})=0$, so 
$\alpha_2=2^{3-1+1-1}(1-2^{-2})2^2=2^4(1-2^{-2})$.

Combining this, we get

$$ Vol(\mathrm{O}(L)) = \frac{d^2}{\pi^2} \prod \limits_{p\nmid d}(1-p^{-2})\prod\limits_{p|d, p\neq 2} \frac{1}{2p}(1-p^{-2})^{-1}(1\pm p^{-1}) \begin{cases}
   2^{-4}(1-2^{-2})^{-1} & \text{for even lattices}\\
   2^{-3} & \text{for odd lattices}
 \end{cases} 
$$
$$ Vol(\mathrm{O}(L))=\begin{cases}
  \frac{1}{3\cdot 2^{\omega(d)+2}} \prod \limits_{p | d, p\neq 2} (p \pm 1 ) & \text{for even lattices}\\
   \frac{1}{ 2^{\omega(d)+3}} \prod \limits_{p | d, p\neq 2} (p \pm 1 )  & \text{for odd lattices}
 \end{cases} $$
where $\omega(d)$ is the number of prime divisors of $d$. 

Consider the case of odd $d$.

All densities except $\alpha_2$ are the same as in the previous case.

The lattice is unimodular, $\mathrm{rank}(N_0^{\mathrm{even}})=2$, $E_{-1}=E_1=\frac{1}{2}$, $E_0=\frac{1}{2}(1\pm 2^{-1})$, the sign plus is chosen iff $N_0^{\mathrm{even}}$ is a hyperbolic plane.
$N_0^{\mathrm{even}}$ is a hyperbolic plane iff $\varepsilon_2(L)=\varepsilon_2(U\oplus <-d>)$ (since the determinants of the left and the right hand sides coincide and are equal to $d$). $\varepsilon_2(U\oplus <-d>)=(-1)^{\epsilon(-d)}$, 
$\epsilon(-d)=
\begin{cases} 
0 \mbox { for } d=-1 \mod 4 \\
1 \mbox { for  } d=1 \mod 4
\end{cases}$.

So the sign plus is chosen when
$d=-1 \mod 4,\ \varepsilon_2(L)=1$ and $d=1 \mod 4, \ \varepsilon_2(L)=-1$. In all other cases the minus sign is chosen.

$\alpha_2=2^{3-1+0-3}(1-2^{-2})\cdot 2^3(1\pm 2^{-1})^{-1} = 2^2(1-2^{-2})(1\pm 2^{-1})^{-1}$.

Combining this, we obtain 
$$Vol(\mathrm{O}(L)) = \frac{d^2}{\pi^2} \prod \limits_{p\nmid d}(1-p^{-2})\prod\limits_{p|d} \frac{1}{2p}(1-p^{-2})^{-1}(1\pm p^{-1})  2^{-2}(1-2^{-2})^{-2}(1\pm 2^{-1})$$
$$Vol(\mathrm{O}(L))=\frac{2\pm 1}{3\cdot 2^{\omega(d)+4}} \prod\limits_{p|d} (p\pm 1)$$
\end{proof}

\section*{Examples}

In this section we compute the volumes of some specific lattices. Fundamental domains for them are known, so we can compare the results obtained by our computation with the areas of the fundamental domains and make sure they coincide. We note that in case of signature $(2,1)$ the volume of $X^V$ is equal to $4\pi$. For the general case we refer to \cite{gri}.

\subsection*{L=$\langle 2 \rangle \oplus U$}
The lattice is even, the discriminant is even, $\omega(d)=1$. Using theorem 5, we get $\mathrm{ Vol_{HM}}(L)=\frac{1}{3\cdot 2^3}$. On the other hand, the fundamental domain for the orthogonal group of this lattice is a triangle with angles $\frac{\pi}{\infty}, \frac{\pi}{3}, \frac{\pi}{2}$. So, the area of this triangle is $\frac{\pi}{6}$. Hence the normalised volume is $\frac{\pi}{6} \cdot \frac{1}{4\pi}=\frac{1}{24}$ and that is exactly what we need.

\subsection*{L=$\langle 6 \rangle \oplus U$}
The lattice is even, the discriminant is even, $\omega(d)=2$. Using theorem 5, we get $\mathrm{ Vol_{HM}}(L)=\frac{1}{3\cdot 2^4} (3+1)$. We notice that the unimodular component is equal to the hyperbolic plane over $\mathbb{Z}_3$, so it represents $0$. That's why we choose the plus sign. On the other hand, the fundamental domain for the orthogonal group of this lattice is a triangle with angles $\frac{\pi}{\infty}, \frac{\pi}{6}, \frac{\pi}{2}$. So, the area of this triangle is $\frac{\pi}{3}$. Hence the normalised volume is $\frac{\pi}{3}\cdot \frac{1}{4\pi} =\frac{1}{12}$ as well.

\subsection*{L=$\langle 4 \rangle \oplus U$} 
We notice that the discriminant of this lattice is not square free. It means that this lattice embeds into a maximal lattice $L'=\langle 1 \rangle \oplus U$. It is unimodular and odd. Using theorem 5, we get $\mathrm{Vol_{HM}}(L')=\frac{2+1}{3\cdot 2^4}$. $N_0^{\mathrm{even}}=U$, that's why we choose the plus sign. The fundamental domain of $O(L)$ is a triangle with angles $\frac{\pi}{\infty}, \frac{\pi}{4}, \frac{\pi}{2}$. It's area is $\frac{\pi}{4}$. Hence the normalised volume of $L$ is $\frac{\pi}{4}\cdot \frac{1}{4\pi}=\frac{1}{16}$. It coincides with the normalised volume of $L$, so the genus of $L$ coincides with it's class. 

\subsection*{L=$6x^2-y^2-z^2$}
The discriminant $d=6$, the lattice is odd, $\omega(d)=2$. Using theorem 5, we get $\mathrm{Vol_{HM}}(L)=\frac{3-1}{2^5}$. 
$-y^2-z^2$ doesn't represent $0$ over $\mathbb{Z}_3$, because it's discriminant is equal to $1$. 
The discriminant of $U$ equals $-1$, $\left( \frac{-1}{3} \right) =-1$, so it can't be equivalent to a hyperbolic plane. That's why we choose minus sign. The fundamental domain is a quadrangle with angles $\frac{\pi}{2},\ \frac{\pi}{2},\ \frac{\pi}{2},\ \frac{\pi}{4}$. It's area is $\frac{\pi}{4}$. Hence the normalised volume is $\frac{\pi}{4} \cdot \frac{1}{4\pi}= \frac{1}{16}$. That is exactly what we need.

\subsection*{L=$11x^2-y^2-z^2$}
The discriminant $d=11$, the lattice is odd, $\omega(d)=1$. We get $\mathrm{Vol_{HM}}(L)=\frac{(2-1)(11-1)}{3\cdot 2^5}$. $-y^2-z^2$ doesn't represent $0$ over $\mathbb{Z}_{11}$, because it's discriminant is equal to $1$. The discriminant of $U$ is $-1$, $\left( \frac{-1}{11} \right) =-1$. That means that $U$ is not equivalent to $-y^2-z^2$, so we choose the minus sign for $11$. We notice that $d=11=-1 \ \mathrm{mod} \ 4 $, $\epsilon_2(L)=(-1,\ -1)_2 = -1$, so we choose the minus sign for $2$ as well. The fundamental domain is a quadrangle with angles $\frac{\pi}{2},\ \frac{\pi}{2},\ \frac{\pi}{3},\ \frac{\pi}{4}$. It's area is $\frac{5\pi}{12}$. Hence the normalised volume is $\frac{5\pi}{12}\cdot \frac{1}{4\pi}=\frac{5}{48}$ as well.

\subsection*{L=$3x^2-5y^2-z^2$}
The lattice is odd, $d=15$, $\omega(d)=2$. Using theorem 5, we get $\mathrm{Vol_{HM}}(L)=\frac{2-1}{3\cdot 2^6}(3+1)(5-1)$. The determinant of the unimodular component is equal to $-3$ over $\mathbb{Z}_5$, $\left( \frac{-3}{5} \right) =-1$, $\mathrm{det}(U)=-1,\ \left( \frac{-1}{5} \right) =1$, so they are not equivalent. That's why we choose the minus sign for $5$. The determinant of the unimodular component is equal to $5$ over $\mathbb{Z}_3$, $\left( \frac{5}{3} \right) =\left( \frac{3}{5} \right) = -1$, $\mathrm{det}(U)=-1,\ \left( \frac{-1}{3} \right) =-1$, so they are equivalent. That's why we choose the plus sign for $3$. The fundamental domain is a quadrangle with angles $\frac{\pi}{2},\ \frac{\pi}{2},\ \frac{\pi}{2},\ \frac{\pi}{6}$, so it's area is $\frac{\pi}{3}$.  Hence the normalised volume is  $\frac{\pi}{3}\cdot \frac{1}{4\pi}=\frac{1}{12}$ as required.

\subsection*{L=$\begin{pmatrix}
\ 2 & -2 & \ 0 \\
-2 & \ 2 & -1 \\
\ 0 & -1 & \ 2 \\
\end{pmatrix}$}

The lattice is even, $d=-2$, $\omega(d)=1$. Using theorem 5, we get $\mathrm{Vol_{HM}}(L)=\frac{1}{3\cdot 2^3}$. Then not normalised volume is $\frac{\pi}{6}$. The fundamental domain is a triangle with angles $\frac{\pi}{2},\ \frac{\pi}{3},\ \frac{\pi}{\infty}$. It's area is $\frac{\pi}{6}$. Hence the normalised volume is $\frac{\pi}{6}\cdot \frac{1}{4\pi}=\frac{1}{24}$ as we need.

\section{Acknowledgements}
Im am very grateful to O.V. Schwarzman for his great ideas, support and inspiration.

\end{document}